\documentclass[]{article}

\usepackage{amsfonts,amssymb,amsbsy,latexsym,amsmath,tabulary,
graphicx,times,caption,fancyhdr,yfonts,amscd,euscript,
calligra,mathrsfs,hyperref,url}
\usepackage[utf8]{inputenc}
\usepackage{url,multirow,morefloats,floatflt,cancel,textcomp,tfrupee}
\usepackage{pifont}
\usepackage[nointegrals]{wasysym}
\usepackage{pifont}

\newtheorem{theorem}{Theorem}[section]
\newtheorem{definition}[theorem]{Definition}
\newtheorem{notation}[theorem]{Notation}

\newtheorem{lemma}[theorem]{Lemma}
\newtheorem{corollary}[theorem]{Corollary}

\newtheorem{remark}[theorem]{Remark}

\newenvironment{proof}{\noindent\mbox{\bf Proof.}}
{\hfill\mbox{\ding{113}}\bigskip}

\makeatletter

\AtBeginDocument{
\expandafter\ifx\csname eqalign\endcsname\relax
\def\eqalign#1{\null\vcenter{\def\\{\cr}\openup\jot\m@th
  \ialign{\strut$\displaystyle{##}$\hfil&$\displaystyle{{}##}$\hfil
      \crcr#1\crcr}}\,}
\fi
}

\let\lt=<
\let\gt=>
\def\processVert{\ifmmode|\else\textbar\fi}

\@ifundefined{subparagraph}{
\def\subparagraph{\@startsection{paragraph}{5}{2\parindent}{0ex plus 0.1ex minus 0.1ex}%
{0ex}{\normalfont\small\itshape}}%
}{}

\newcommand\role[1]{\unskip}
\newcommand\aucollab[1]{\unskip}

\@ifundefined{tsGraphicsScaleX}{\gdef\tsGraphicsScaleX{1}}{}
\@ifundefined{tsGraphicsScaleY}{\gdef\tsGraphicsScaleY{.9}}{}
\def\checkGraphicsWidth{\ifdim\Gin@nat@width>\textwidth
	\tsGraphicsScaleX\textwidth\else\Gin@nat@width\fi}

\def\checkGraphicsHeight{\ifdim\Gin@nat@height>.9\textheight
	\tsGraphicsScaleY\textheight\else\Gin@nat@height\fi}

\def\fixFloatSize#1{\@ifundefined{processdelayedfloats}{\setbox0=\hbox{\includegraphics{#1}}\ifnum\wd0<\columnwidth\relax\renewenvironment{figure*}{\begin{figure}}{\end{figure}}\fi}{}}
\let\ts@includegraphics\includegraphics

\def\inlinegraphic[#1]#2{{\edef\@tempa{#1}\edef\baseline@shift{\ifx\@tempa\@empty0\else#1\fi}\edef\tempZ{\the\numexpr(\numexpr(\baseline@shift*\f@size/100))}\protect\raisebox{\tempZ pt}{\ts@includegraphics{#2}}}}

\AtBeginDocument{\def\includegraphics{\@ifnextchar[{\ts@includegraphics}{\ts@includegraphics[width=\checkGraphicsWidth,height=\checkGraphicsHeight,keepaspectratio]}}}

\def\URL#1#2{\@ifundefined{href}{#2}{\href{#1}{#2}}}

\def\UrlOrds{\do\*\do\-\do\~\do\'\do\"\do\-}%
\g@addto@macro{\UrlBreaks}{\UrlOrds}
\makeatother

\makeatletter

\def\wileyIndent{1pt}
\usepackage[paperheight=10in,paperwidth=6.5in,margin=2cm,headsep=.5cm,top=2.5cm]{geometry}

\renewenvironment{abstract}
{\vspace*{-1pc}\trivlist\item[]\leftskip\wileyIndent\hrulefill\par\vskip4pt\noindent\textbf{\abstractname}\mbox{\null}\\}{\par\noindent\hrulefill\endtrivlist}

\def\author#1{\gdef\@author{\hskip-\dimexpr(\tabcolsep)\hskip\wileyIndent\parbox{\dimexpr\textwidth-\wileyIndent}{\centering\bfseries#1}}}

\def\title#1{\gdef\@title{\centering\bfseries\ifx\@articleType\@empty\else\@articleType\\\fi#1}}

\let\@articleType\@empty \def\articletype#1{\gdef\@articleType{{\normalfont\itshape#1}}}

\fancypagestyle{headings}{\fancyhf{}\fancyhead[C]{\RunningHead}\fancyhead[R]{\thepage}}

 \def\audegree#1{}

\date{}

\makeatother

\usepackage[T1]{fontenc}
\makeatother
\usepackage[numbers,sort&compress]{natbib}

\begin{document}

\title{On Constructivity and the Rosser Property: \\  a closer look at some G\"odelean proofs}

\author{\textsc{Saeed Salehi \,\&\,  Payam Seraji}}

\def\RunningHead{\textsc{\small S.~Salehi \& P.~Seraji} / {\sl On Constructivity and the Rosser Property \; }}
\def\RunningAuthor{S.~Salehi \& K.~Lajevardi}


\maketitle

\begin{abstract}
The proofs of  Kleene, Chaitin and Boolos for G\"odel's First Incompleteness Theorem are studied from the perspectives of constructivity and  the Rosser property. A proof of the incompleteness theorem has the Rosser property when the independence of the true but unprovable sentence can be shown by assuming only the (simple) consistency of the theory. It is known that G\"odel's own proof for his incompleteness theorem does not have the Rosser property, and we show that neither do Kleene's or Boolos' proofs. However, we show that a variant of Chaitin's proof can have the Rosser property. The proofs of G\"odel, Rosser and Kleene are constructive in the sense that they explicitly construct, by algorithmic ways, the independent sentence(s) from the theory. We show that the proofs of Chaitin and Boolos are not constructive, and they prove only the mere existence of the independent sentences.
\end{abstract}

\section*{} 

\section{Introduction}\label{intro}
A constructive proof provides an algorithm for constructing the claimed object; a non-constructive proof does not show the existence of that object algorithmically, even if sometimes  an effective procedure might be hidden inside the details. A proof then {\em is proved to be \textup{(}essentially\textup{)}   non-constructive} when one can show that there is no algorithm (computable function) which, given the assumptions (coded as input), produces the claimed object whose existence is demonstrated in the proof. Below, we will see one example of a (seemingly) non-constructive proof (namely, the proof of Kleene~\cite{kleene1} for G\"odel's first incompleteness theorem, stated below) which can be made constructive by unpacking some details; we will also see a couple of proofs (namely, the proofs of Boolos~\cite{boolos} and Chaitin~\cite{chai} for G\"odel's first incompleteness theorem) that are shown to be non-constructive, by proving the non-existence of any algorithm for computing the claimed object (namely, the true but unprovable sentence).

\noindent The (First) {\em Incompleteness Theorem} (of G\"odel~\cite{godel}) states that for a  sufficiently strong   {\sc re} theory $T$   there exists a sentence $\boldsymbol\psi_T$ in the language of $T$ such that
\begin{enumerate}
\item the sentence $\boldsymbol\psi_T$ is true (in the standard model of natural numbers);

\item  if $T$ is consistent then $T\nvdash \boldsymbol\psi_T$;

\item  if $T$ is $\omega$-consistent then $T\nvdash \neg\boldsymbol\psi_T$.
\end{enumerate}

\noindent
By {\em a proof of the incompleteness theorem} we mean a demonstration of the existence of such a sentence ($\boldsymbol\psi_T$) for any given consistent and {\sc re} theory $T$ that is  sufficiently strong (to be made precise later).
Such a proof witnesses  {\em  the Rosser property} (\cite{rosser}) when the condition of $\omega$-consistency can be replaced with (simple) consistency; that is to say that the condition 3 above can be replaced with the following condition
\begin{enumerate}
\item[3$^\prime$.] if $T$ is consistent then $T\nvdash \neg\boldsymbol\psi_T$.
\end{enumerate}

\noindent
G\"odel's original proof~\cite{godel} for his incompleteness theorem is constructive, i.e., given a (finite) description of a consistent {\sc re} theory (e.g.\ an input-free program which outputs the set of all the axioms of the theory) the proof exhibits, in an algorithmic way, a sentence which is true (in the standard model of natural numbers $\mathbb{N}$) but unprovable in the theory. For the independence of this sentence from the theory (i.e., the unprovability of its negation in the theory) G\"odel also assumes the theory to be  $\omega$-consistent; so if the theory is 
 $\omega$-consistent, then that (true) sentence is independent from the theory (see e.g.\ \cite{smor,smith}). It turned out later that the simple consistency of the theory does not suffice for the independence of the G\"odel sentence (from the theory) and the optimal condition (which is  much weaker than $\omega$-consistency) is \textsl{the consistency of the theory with its own consistency statement} (\cite[Theorems~35,36]{isaac}).
Rosser's  proof \cite{rosser} for G\"odel's first incompleteness theorem assumes only the simple consistency of the ({\sc re}) theory and constructs (algorithmically) an independent (and true) sentence. So, one can say that G\"odel's proof does  not have the Rosser property. Here, we will see that while a variant of the proof of Chaitin has the Rosser property (i.e., the independence of Chaitin's sentence from the theory can be proved by assuming only the simple consistency of the theory), the proof of Boolos does not have the Rosser property (and the optimal condition for the independence of a Boolos  sentence is the consistency of the theory with its own consistency statement).

\section{Kleene's Proof for G\"odel's Incompleteness Theorem}
A very cute proof for G\"odel's incompleteness theorem is that of Kleene (see e.g.\ \cite{kleene1,smith}) which deserves more recognition.

\begin{notation}[Computability] {\rm
Let $\boldsymbol\varphi_0,\boldsymbol\varphi_1,\boldsymbol\varphi_2,
\cdots$ be a list of all unary computable (partial recursive) functions (in a way that $\boldsymbol\varphi_i(j)$, if it exists, can be computed from $i$ and $j$). A recursively enumerable set  ({\sc re} for short) is the domain of $\boldsymbol\varphi_i$, for some $i\in\mathbb{N}$, which is denoted by $\mathcal{W}_i$. The notation $\boldsymbol\varphi_i(j)\!\!\uparrow$ means that the function $\boldsymbol\varphi_i$ is not defined at $j$, or $j\not\in\mathcal{W}_i$; and   $\boldsymbol\varphi_i(j)\!\!\downarrow$ means that $\boldsymbol\varphi_i$ is  defined at $j$ or $j\in\mathcal{W}_i$. Needless to say,  $\boldsymbol\varphi_i(j)=k$   means that $\boldsymbol\varphi_i$ is defined at $j$ and is equal  to $k$.}
\hfill \ding{71}
\end{notation}

\noindent
Robinson's Arithmetic is denoted by $\textsl{\textsf{Q}}$ (see~\cite{tmr} or~\cite{smith}). In all the results of this paper, the theory $\textsl{\textsf{Q}}$ can be replaced with a (much) weaker theory called $\textsl{\textsf{R}}$ (see~\cite{tmr}). The theory $\textsl{\textsf{Q}}$ is finitely axiomatizable, while $\textsl{\textsf{R}}$ is not.

\begin{theorem}[Kleene's Theorem]\label{kleeneth}
For a given consistent and {\sc re} theory $T$ that contains $\textsl{\textsf{Q}}$ there exists some $t\in\mathbb{N}$ such that $\boldsymbol\varphi_t(t)\!\!\uparrow$ but $T\nvdash``\boldsymbol\varphi_t(t)\!\!\uparrow\!\!"$.
\end{theorem}

\begin{proof} (Non-Constructive Proof): 
Let $\mathcal{K}_T=\{n\in\mathbb{N}\mid T\vdash``\boldsymbol\varphi_n(n)\!\!\uparrow\!\!"\}$; then we have $\mathcal{K}_T\subseteq K=\{n\in\mathbb{N}\mid \boldsymbol\varphi_n(n)\!\!\uparrow\}$ since if $T\vdash``\boldsymbol\varphi_n(n)\!\!\uparrow\!\!"$ but $\boldsymbol\varphi_n(n)\!\!\downarrow$ then the true $\Sigma_1$-sentence $``\boldsymbol\varphi_n(n)\!\!\downarrow\!\!"$ is provable in (the $\Sigma_1$-complete theory) $\textsl{\textsf{Q}}$ ($\subseteq T$) contradicting the consistency of $T$. Now, since $T$ is {\sc re} then so is $\mathcal{K}_T$, while $K$ is not an {\sc re} set because for any $n$ we have $n\in K \iff n\not\in\mathcal{W}_n$ and so $n\in K\boldsymbol\triangle\mathcal{W}_n$, thus $K\neq\mathcal{W}_n$ for all $n$. So, $\mathcal{K}_T\subsetneqq K$; therefore, there must exist some $t\in K-\mathcal{K}_T$. For this $t$ we have $\boldsymbol\varphi_t(t)\!\!\uparrow$ but $T\nvdash``\boldsymbol\varphi_t(t)\!\!\uparrow\!\!"$.
\end{proof}

\noindent 
Of course if $T$ is sound (i.e., $\mathbb{N}\models T$) or even $\Sigma_1$-sound (i.e., if $\sigma\in\Sigma_1$ and $T\vdash\sigma$ then $\mathbb{N}\models\sigma$, cf.~\cite{isaac}) then  also  $T\nvdash``\boldsymbol\varphi_t(t)\!\!\downarrow\!\!"$, i.e., the sentence $``\boldsymbol\varphi_t(t)\!\!\uparrow\!\!"$ is (true and) independent from $T$. Let us note that the above proof did not explicitly specify $t\in\mathbb{N}$.

\bigskip 

 \begin{proof} (Constructive Proof): 
Since  $\mathcal{K}_T=\{n\in\mathbb{N}\mid T\vdash``\boldsymbol\varphi_n(n)\!\!\uparrow\!\!"\}$ is {\sc re} then $\mathcal{K}_T=\mathcal{W}_\texttt{t}$ for some $\texttt{t}\in\mathbb{N}$ which can be algorithmically computed from a description of the {\sc re} theory $T$. Now we show the truth of $``\boldsymbol\varphi_\texttt{t}(\texttt{t})\!\!\uparrow\!\!"$ as follows:

\begin{tabular}{llll}
$\boldsymbol\varphi_\texttt{t}(\texttt{t})\!\!\downarrow$ & $\Longrightarrow$ & $T\vdash``\boldsymbol\varphi_\texttt{t}(\texttt{t})\!\!\downarrow\!\!"$ & (by the $\Sigma_1$-completeness of $\textsl{\textsf{Q}}\subseteq T$)\\
 & $\Longrightarrow$ & $T\nvdash``\boldsymbol\varphi_\texttt{t}(\texttt{t})\!\!\uparrow\!\!"$ & (by the consistency of $T$)\\
 & $\Longrightarrow$ & $\texttt{t}\not\in \mathcal{K}_T$ & (by the definition of $\mathcal{K}_T$)\\
  & $\Longrightarrow$ & $\texttt{t}\not\in\mathcal{W}_\texttt{t}$ & (by $\mathcal{K}_T=\mathcal{W}_\texttt{t}$)\\
 & $\Longrightarrow$ & $\boldsymbol\varphi_\texttt{t}(\texttt{t})\!\!\uparrow$ & (by the definition of $\mathcal{W}_\texttt{t}$)
\end{tabular}

\noindent Thus, $\texttt{t}\not\in\mathcal{W}_\texttt{t}$ and so $\texttt{t}\not\in \mathcal{K}_T$ whence $T\nvdash``\boldsymbol\varphi_\texttt{t}(\texttt{t})\!\!\uparrow\!\!"$.
\end{proof}

\noindent 
Indeed, for any {\sc re} and consistent theory $T$($\supseteq\textsl{\textsf{Q}}$) and {\em any} $t$ with $\mathcal{W}_t=\mathcal{K}_T$ we have (by the above proof) that $\boldsymbol\varphi_t(t)\!\!\uparrow$ and $T\nvdash``\boldsymbol\varphi_t(t)\!\!\uparrow\!"$.
 Below we show that Kleene's (constructive) proof does not have the Rosser property.

\begin{theorem}[Kleene's Proof is not Rosserian]\label{nrk} For any given consistent and {\sc re} theory $T\supseteq\textsl{\textsf{Q}}$ there exists an {\sc re} and  consistent   theory $U\supseteq T$ such that $U\vdash``\boldsymbol\varphi_u(u)\!\!\downarrow\!\!"$ for some $u\in\mathbb{N}$ which satisfies $\mathcal{W}_u=\{n\in\mathbb{N}\mid U\vdash``\boldsymbol\varphi_n(n)\!\!\uparrow\!\!"\}$
\textup{(}and $\boldsymbol\varphi_u(u)\!\uparrow$\textup{)}.
\end{theorem}

\begin{proof}
There exists a computable (and total) function $\hbar$ such that for any sentence $\psi$ we have  $\mathcal{W}_{\hbar({\psi})}=\{n\in\mathbb{N}\mid T+\psi\vdash``\boldsymbol\varphi_n(n)\!\!\uparrow\!\!"\}$. By the Diagonal Lemma there exists a sentence $\lambda$ such that $\textsl{\textsf{Q}}\vdash \lambda\leftrightarrow ``\boldsymbol\varphi_{\hbar({\lambda})}
\big(\hbar({\lambda})\big)\!\!\downarrow\!\!"$. Clearly, for the theory  $U=T+\lambda$ and $u=\hbar({\lambda})$ we have $U\vdash``\boldsymbol\varphi_u(u)\!\!\downarrow\!\!"$ and $\mathcal{W}_u=\{n\in\mathbb{N}\mid U\vdash``\boldsymbol\varphi_n(n)\!\!\uparrow\!"\}$. It remains to show that $U$ is consistent: Otherwise,  $T\vdash\neg\lambda$ and so $T\vdash``\boldsymbol\varphi_u(u)\!\!\uparrow\!"$ which implies that $T+\lambda\vdash``\boldsymbol\varphi_u(u)\!\!\uparrow\!"$ whence $u\in\mathcal{W}_{\hbar(\lambda)}=\mathcal{W}_u$.
On the other hand $T\vdash``\boldsymbol\varphi_u(u)\!\!\uparrow\!"$ implies that $\boldsymbol\varphi_u(u)\!\!\uparrow$ holds (since otherwise $\boldsymbol\varphi_u(u)\!\!\downarrow$ by the $\Sigma_1-$completeness would imply $T\vdash``\boldsymbol\varphi_u(u)\!\!\downarrow\!\!"$ contradicting the consistency of $T$) and so
$u\not\in\mathcal{W}_u$; a contradiction.
\end{proof}

\noindent 
Summing up, for any consistent and {\sc re} extension $T$ of $\textsl{\textsf{Q}}$ we have $T\nvdash``\boldsymbol\varphi_{t}({t})\!\!\uparrow\!\!"$ and  $\boldsymbol\varphi_{t}({t})\!\!\uparrow$ for any ${t}$ which satisfies $\mathcal{W}_{t}=\{n\in\mathbb{N}\mid T\vdash``\boldsymbol\varphi_n(n)\!\!\uparrow\!\!"\}$. Moreover, if $T$ is $\Sigma_1$-sound then $\boldsymbol\varphi_{t}({t})\!\!\uparrow$ is independent from $T$ (i.e., we also have  $T\nvdash``\boldsymbol\varphi_{t}({t})\!\!\downarrow\!\!"$). However, if the theory $T$ is not $\Sigma_1$-sound then for some $e$ with $\mathcal{W}_e=\{n\in\mathbb{N}\mid T\vdash``\boldsymbol\varphi_n(n)\!\!\uparrow\!"\}$ the sentence  $\boldsymbol\varphi_e(e)\!\!\uparrow$ might not be independent from $T$ (and its negation could be provable in $T$, that is $T\vdash``\boldsymbol\varphi_{e}({e})\!\!\downarrow\!\!"$).

\noindent 
Albert Visser has informed the authors that for any {\sc re} and consistent theory $T$ which is sufficiently strong (see e.g. the explanations before Theorem~\ref{nrobo} below) there exists some $\vartheta$ with $\mathcal{W}_\vartheta=\{n\in\mathbb{N}\mid T\vdash``\boldsymbol\varphi_n(n)\!\!\uparrow\!"\}$ such that (beside $\boldsymbol\varphi_{\vartheta}({\vartheta})\!\!\uparrow$ and $T\nvdash
``\boldsymbol\varphi_{\vartheta}({\vartheta})\!\!\uparrow\!\!"$ we also have) $T\nvdash``
\boldsymbol\varphi_{\vartheta}({\vartheta})\!\!\downarrow\!\!"$, or in the other words the sentence $\boldsymbol\varphi_\vartheta(\vartheta)\!\!\uparrow$ is independent from $T$; moreover $\vartheta$ can be algorithmically computed from a given description of the {\sc re} theory $T$. The proof of this Rosserian version of Kleene's proof is rather involved and will appear in a future paper.
Let us note that a Rosserian version of this beautiful theorem of Kleene appeared in \cite{kleene2} (see also \cite{kleene3}) where Kleene calls it  ``a symmetric form'' of G\"odel's (incompleteness) theorem (also see \cite{salehi} for a modern treatment).

\section{Chaitin's Proof for G\"odel's Incompleteness Theorem}
There are various versions of Chaitin's proof for the incompleteness theorem \cite{chai}, which is sometimes called ``Chaitin's incompleteness theorem''; this proof appears in e.g.\  \cite{davis,lambal,bek,stillwell}. We consider the version presented  in \cite{bek}.
\begin{definition}[Kolmogorov-Chaitin Complexity] For any natural number $m$ let $\mathscr{K}(m)=\min \{i\in\mathbb{N}\mid \boldsymbol\varphi_i(0)\!\!\downarrow =m\}$.
\hfill \ding{71}\end{definition}

\noindent 
The function $\mathscr{K}$ is total and for any $e\in\mathbb{N}$ there are finitely many $m$'s which satisfy  $\mathscr{K}(m)\leqslant e$. The following is Lemma~7 of~\cite{bek}.
\begin{lemma}[Uncomputability of Complexity]\label{mainlemma}
There is no computable function $f$ which satisfies $\mathscr{K}\big(f(m)\big)> m$ for all $m\in\mathbb{N}$.
\end{lemma}

\begin{proof}
If there were such a computable function $f$, then by Kleene's second recursion theorem there would exist some $e$ such that $\boldsymbol\varphi_e(x)=f(e)$ and so, 
$\boldsymbol\varphi_e(0)=f(e)$ which implies $\mathscr{K}\big(f(e)\big)\leqslant e$; a  contradiction.
\end{proof}

\noindent 
So,   $\mathscr{K}$ is not computable, since otherwise   $f(x)=\min \{y \mid \mathscr{K}(y)>x\}$, which satisfies $\forall x: \mathscr{K}\big(f(x)\big)> x$, would be computable.
\begin{theorem}[Chaitin's Theorem]\label{chth}
For any consistent  {\sc re} theory $T$ which contains $\textsl{\textsf{Q}}$ there exists a constant $c_T\in\mathbb{N}$ such that for any $e\geqslant c_T$ and any $w\in\mathbb{N}$ we have $T\nvdash``\mathscr{K}(w)>e"$.
\end{theorem}

\begin{proof} If not, then for any given $m\in\mathbb{N}$ there exists some $e\geqslant m$ and some $w$ such that $T\vdash``\mathscr{K}(w)>e"$. Let us note that if $T\vdash``\mathscr{K}(w)>e"$ for a consistent $T\supseteq\textsl{\textsf{Q}}$ then $\mathscr{K}(w)>e$, since otherwise, if $\mathscr{K}(w)\leqslant e$, the true $\Sigma_1$-sentence $``\mathscr{K}(w)\leqslant e"$ would be provable in $\textsl{\textsf{Q}}$ (and so in $T$) which contradicts the consistency of $T$. Now, for a given $m$ we can, by an algorithmic proof search in $T$, find some $e\geqslant m$ and $w$ such that $T\vdash``\mathscr{K}(w)>e"$ (and so $\mathscr{K}(w)>e$); our assumption guarantees the termination of this algorithm for any input $m$. Let $f(m)$ be one of those $w$'s; then 
$\mathscr{K}\big(f(m)\big)>e\geqslant m$   which contradicts   Lemma~\ref{mainlemma}.
\end{proof}

\noindent 
This is an incompleteness theorem since for any $c$ there are cofinitely many $w$'s with $\mathscr{K}(w)>c$. So, for a given $T$ which is consistent and {\sc re} and contains $\textsl{\textsf{Q}}$ there are cofinitely many $w$'s such that the true sentences $``\mathscr{K}(w)>c_T"$ are unprovable in $T$. As for the constructivity of this proof, the good news is that a constant $c_T$ which satisfies Chaitin's Theorem (\ref{chth}) can be algorithmically constructed from $T$.
\begin{theorem}[Computing a Chaitin Constant]\label{thct}
For a given consistent and {\sc re} extension $T$ of $\textsl{\textsf{Q}}$ one can algorithmically construct a constant $c_T$ such that for all $e\geqslant c_T$ and all $w$, we have $T\nvdash``\mathscr{K}(w)>e"$.
\end{theorem}

\begin{proof}
Given a description of a consistent, $\Sigma_1$-complete and {\sc re} theory $T$ the following can be done algorithmically. Define $\hbar(x,y)$ to be the first   ordered pair    $\langle a,b\rangle$ such that the proof search algorithm of $T$ shows up (a proof of) the sentence $``\mathscr{K}(a)>b\geqslant x"$ (so, $T\vdash ``\mathscr{K}(a)>b\geqslant x"$).  This is (a partially) computable (function) and an index of it  can be calculated from (a description of) $T$.  By  Kleene's second recursion theorem there exists a constant $c$ such that $\boldsymbol\varphi_c(y)=\hbar_1(c,y)$, where $\hbar_1(x,y)$ is the first component of the ordered pair $\hbar(x,y)$. The constant $c$ can be computed from $T$; let us denote it by $c_T$. Now, we show that for no $b\geqslant c_T$ and no $a$ can $T\vdash``\mathscr{K}(a)>b"$ hold. If there exists such $a$ and $b$ then we have $T\vdash``\mathscr{K}(a)>b\geqslant c_T"$. If $\langle a,b\rangle$ is the first ordered  pair such that $``\mathscr{K}(a)>b\geqslant c_T"$ appears in the above mentioned
proof search algorithm of $T$,
  then $\hbar(c_T,0)=\langle a,b\rangle$ and so $\boldsymbol\varphi_{c_T}(0)=\hbar_1(c_T,0)=a$. Thus, $\mathscr{K}(a)\leqslant c_T$, and by the $\Sigma_1$-completeness of $T$ we have $T\vdash``\mathscr{K}(a)\leqslant c_T"$. But from $T\vdash``\mathscr{K}(a)>b\geqslant c_T"$ we have $T\vdash``\mathscr{K}(a)>c_T"$,   contradicting the consistency of $T$.
\end{proof}

\noindent 
Unfortunately, by Lemma~\ref{mainlemma} one cannot calculate a $w$ with $\mathscr{K}(w)>c_T$ given $c_T$ for a theory $T$. Otherwise one could get a constructive version of Chaitin's proof: Given a consistent and {\sc re} theory $T\supseteq\textsl{\textsf{Q}}$ one calculates $c_T$ and finds some $w$ with $\mathscr{K}(w)>c_T$; then $``\mathscr{K}(w)>c_T"$ is a true sentence which is not provable in $T$. It is actually known that Chaitin's proof is not constructive; see e.g.\ \cite[page~1394]{lambal} or  \cite[page~95]{stillwell}.

\begin{theorem}[Non-Constructivity of Chaitin's Proof]\label{unconch}
There is no algorithm such that for a given consistent and {\sc re} extension $T$ of $\textsl{\textsf{Q}}$ can compute some $w_T$ such that both $T\nvdash``\mathscr{K}(w_T)>c_T"$ and
$\mathscr{K}(w_T)>c_T$ holds, where $c_T$ is a Chaitin constant as in Theorem~\ref{thct}.
\end{theorem}

\begin{proof}
If such a $w_T$ were computable from $T$, then
the theory $T_\infty=\bigcup_{i\in\mathbb{N}} T_i$ would be {\sc re} where $T_0=\textsl{\textsf{Q}}$ and inductively $T_{i+1}=T_i+``\mathscr{K}(w_{T_i})>c_{T_i}"$ are defined
by iterating the computation procedure. The theory $T_\infty$ is also consistent (indeed, sound) and contains $\textsl{\textsf{Q}}$, so by Chaitin's Theorem (\ref{chth}) there should exist some constant $c_{T_\infty}$ such that for no $w$ can we have the deduction  $T_\infty\vdash``\mathscr{K}(w)>c_{T_\infty}"$. But this is a contradiction  because we have  $c_{T_i}<c_{T_{i+1}}$ and also $c_{T_i}<c_{T_{\infty}}$ for all $i\in\mathbb{N}$.
\end{proof}

\begin{proof} (An Alternative Proof) 
 Albert Visser suggested the following argument as another proof of Theorem~\ref{unconch}: Since the sequence $\{c_{T_i}\}_{i\in\mathbb{N}}$ is strictly increasing, we have that  $c_{T_m}\geqslant m$ for any $m\in\mathbb{N}$. Now, $\forall m\in\mathbb{N}:\mathscr{K}(w_{T_m})>c_{T_m}\geqslant m$ would contradict Lemma~\ref{mainlemma} if $w_T$ were computable from $T$.
\end{proof}

\begin{remark}\label{referee}{\rm 
Albert Visser noted that Theorems~\ref{thct} and~\ref{unconch} amusingly imply Lemma~\ref{mainlemma}, since if there were a computable (total) function $f$ with $\forall m\in\mathbb{N}: \mathscr{K}\big(f(m)\big)>m$ then one could take $w_T$  as $f(c_T)$.
}\hfill \ding{71}
\end{remark}
\noindent 
The true unprovable sentences $``\mathscr{K}(w)>e"$ (for $e\geqslant c_T$) are also independent when $T$ is a ($\Sigma_1$-)sound theory: If $T\vdash``\mathscr{K}(w)\leqslant e"$ then the $\Sigma_1$-sentence $\mathscr{K}(w)\leqslant e$ has to be true, a contradiction. So, we restate Chaitin's Theorem as

\begin{corollary}[Chaitin's Theorem, restated]\label{corch}
Let $T$ be
 a $\Sigma_1$-sound and {\sc re} theory such that   $T\supseteq\textsl{\textsf{Q}}$.  There exists some $c_T$ \textup{(}which is computable from $T$\textup{)} such that for any $e\geqslant c_T$ there are cofinitely many $w$'s such that $``\mathscr{K}(w)>e"$ is independent from $T$.\hfill \ding{113}
\end{corollary}

\noindent 
For  a Rosserian version of Chaitin's Theorem   the assumption of the ``$\Sigma_1$-soundness'' (of $T$) in Corollary~\ref{corch} should be replaced  with (its simple) ``consistency''. For doing that we need the following  version of the Pigeonhole Principle in $\textsl{\textsf{Q}}$ (which  holds in  $\textsl{\textsf{R}}$ too).

\begin{lemma}[A Pigeonhole Principle]\label{lemphp}
For any $k\in\mathbb{N}$ we have
$$\textsl{\textsf{Q}}\vdash\forall z_0,\cdots,z_k\Big( \quad   \bigwedge\hspace{-4ex}\bigwedge_{0\leqslant i\leqslant k}z_i<\overline{k}\quad\longrightarrow
\quad\bigvee\hspace{-4.75ex}\bigvee_{0\leqslant i,j\leqslant k}^{i\neq j}z_i=z_j\Big).$$
\end{lemma}

\begin{proof} This can be proved by induction (in the metalanguage) on $k$: for $k=0$ it suffices to note that
$\textsl{\textsf{Q}}\vdash\forall z\neg(z<0)$ and for the induction step it suffices to use the derivation  $\textsl{\textsf{Q}}\vdash\forall z (z<\overline{k+1}\rightarrow z<\overline{k}\vee z=\overline{k})$; cf.~\cite[page~73]{smith}.
\end{proof}

\begin{theorem}[Rosserian form of Chaitin's Theorem]\label{rosch}
For any consistent  {\sc re} extension  $T$ of   $\textsl{\textsf{Q}}$ there is a constant $c_T$ \textup{(}which is computable from $T$\textup{)} such that for any $e\geqslant c_T$ there are cofinitely many $w$'s such that $``\mathscr{K}(w)>e"$ is independent from $T$.
\end{theorem}

\begin{proof}
By Chaitin's Theorem (\ref{chth}) there exists a constant $c_T$ (which is computable from $T$) such that for any $e\geqslant c_T$ there are  cofinitely many $w$'s such that  $``\mathscr{K}(w)>e"$ is true but unprovable in $T$. Fix an $\textsf{e}\geqslant c_T$. For no $w$ can $T\vdash``\mathscr{K}(w)>\textsf{e}"$ hold, and
  $T\vdash``\mathscr{K}(w)\leqslant \textsf{e}"$ can hold for at most $(\textsf{e}+1)$-many $w$'s:
if for some distinct $\textsf{w}_0,\textsf{w}_1,\dots,\textsf{w}_{\textsf{e}+1}$, the derivations  $T\vdash``\mathscr{K}(\textsf{w}_i)\leqslant \textsf{e}"$  hold ($i=0,1,\dots,\textsf{e}$+1) then $T\vdash\exists z_0,z_1,\dots,z_{\textsf{e}+1}
\big(\bigwedge\hspace{-2ex}\bigwedge_{i=0}^{\textsf{e}+1}
[z_i\leqslant\overline{\textsf{e}} \wedge \boldsymbol\varphi_{z_i}(0)\!\!\downarrow=\textsf{w}_i]\big)$
and so   $T\vdash\exists z_0,z_1,\dots,z_{\textsf{e}+1}
\big(\bigwedge\hspace{-2ex}\bigwedge_{0\leqslant i \leqslant \textsf{e}+1}
[z_i < \overline{\textsf{e}}+1] \wedge \bigwedge\hspace{-2ex}\bigwedge^{i\neq j}_{0 \leqslant i,j \leqslant \textsf{e}+1}[z_i\neq z_j]\big)$ which contradicts   Lemma~\ref{lemphp} (for $k=\textsf{e}+1$).
Thus, for cofinitely many   $w$'s we should have both  $T\nvdash``\mathscr{K}(w)>\textsf{e}"$ and $T\nvdash``\mathscr{K}(w)\leqslant\textsf{e}"$.
\end{proof}

\noindent 
Martin Davis \cite{davis} calls Chaitin's Theorem ``a dramatic extension of G\"odel's incompleteness theorem''. We saw that this theorem as presented in Corollary~\ref{corch} can be hardly considered   an extension of G\"odel's incompleteness theorem, as G\"odel's proof is constructive while Chaitin's is not (Theorem~\ref{unconch}). The Rosserian  form of Chaitin's Theorem as presented in Theorem~\ref{rosch} could be considered as an extension of G\"odel's and Chaitin's theorems in a sense,  even though, it  is not any more extension than Rosser's own~\cite{rosser}; let us also note that Rosser's proof is constructive (while the proof of Theorem~\ref{rosch} is not).

\section{Boolos'  Proof  for G\"odel's Incompleteness Theorem}
Jon Barwise calls it ``a very lovely proof of G\"odel's Incompleteness Theorem, probably the deepest single
result about the relationship between computers and
mathematics'', and mentions that it is ``the most straightforward proof of this result that I
have ever seen''\footnote{J. Barwise, ``Editorial Notes: This Month's Column'', {\em Notices of the American Mathematical Society}, vol~36, no.~4 (1989), page~388.}. After its first appearance  in \cite{boolos} this proof  was discussed, extended and studied in e.g.\  \cite{kiku,kita,roy,kotlar,sere,kikusa}.

\begin{notation}[Arithmetization]{\rm
For an {\sc re} theory $T$ denote the provability predicate of $T$ by $\textsl{\textsf{Pr}}_T(x)$; so $\textsl{\textsf{Con}}(T)=\neg\textsl{\textsf{Pr}}_T(\bot)$ is the consistency statement of $T$.
Suppose that the variables are $x,x',x'',x''',\cdots$ whose lengths are $1,2,3,4,\cdots$, respectively.
}\hfill \ding{71}
\end{notation}

\noindent 
  So, for any $k\in\mathbb{N}$ there are at most finitely many formulas with length $k$.


\begin{definition}[Formalizing Berry's Paradox]{\rm
For a formula $\psi(x_1,\cdots,x_m)$ with the shown (possibly empty) set of  free variables ($m\geqslant 0$)   and number $n$,  let ${\textsf{D}}(\psi,n)$ be the (G\"odel code of the)  formula $\forall x [\psi(x,\cdots,x)\leftrightarrow x=\overline{n}]$. The number $n$ is definable in $T$ by the formula $\psi$ when $\textsl{\textsf{Pr}}_T\big({\textsf{D}}(\psi,n)\big)$ holds.

Let $\textsl{\textsf{Def}}_T^{\;<z}(y)=\exists x\big[\textsl{\textsf{len}}(x)<z \wedge \textsl{\textsf{Pr}}_T\big({\textsf{D}}(x,y)\big)\big]$, where $\textsl{\textsf{len}}(x)$ denotes the length of (the formula  with G\"odel code) $x$. The  formula $\textsl{\textsf{Def}}_T^{\;<z}(y)$ states  that ``there exists  a formula $\psi(x_1,\cdots,x_m)$ whose length is smaller than the number  $z$ such that the deduction
$T\vdash\forall x[\psi(x,\cdots,x)\leftrightarrow x=\overline{y}]$ holds'', or informally ``the number $y$ is definable in $T$ by a formula with length less than $z$''.

 Let $\textsl{\textsf{Berry}}_T^{\;<v}(u)=
\neg\textsl{\textsf{Def}}_T^{\;<v}(u)
\wedge \forall y\!<u\!\;\textsl{\textsf{Def}}_T^{\;<v}(y)$, meaning that ``$u$ is the least number not definable by a formula with length less than $v$''.

 Let $\ell_T$ be the length of the formula $\textsl{\textsf{Berry}}_T^{\;<x'}(x)$ and let    $\textsl{\textsf{Boolos}}_T(x)$ be the formula $\exists x' \big[ x'= \overline{5}\cdot\overline{\ell_T} \wedge
\textsl{\textsf{Berry}}_T^{\;<x'}(x)\big]$.
Let $\textsl{\textsf{b}}_T$ be  the least number not definable by a formula with length less than $5\ell_T$.
}\hfill \ding{71}
\end{definition}

\begin{theorem}[Boolos' Theorem]\label{both}
For any consistent and {\sc re} extension $T$ of $\textsl{\textsf{Q}}$, the sentence $\textsl{\textsf{Boolos}}_T(\overline{\textsl{\textsf{b}}_T})$ is \textup{(}true but\textup{)} unprovable in $T$.
\end{theorem}

\begin{proof}
First we show that $\textsl{\textsf{Q}}\vdash\forall u,v [\textsl{\textsf{Berry}}_T^{\;<v}(\overline{n})
\wedge\textsl{\textsf{Berry}}_T^{\;<v}(u)
\rightarrow \overline{n}=u]$ holds for any $n\in\mathbb{N}$. Reason inside $\textsl{\textsf{Q}}$:  if for some $u,v$ we have (a)~$\textsl{\textsf{Berry}}_T^{\;<v}(\overline{n})$ and (b)~$\textsl{\textsf{Berry}}_T^{\;<v}(u)$ then (a')~$\neg\textsl{\textsf{Def}}_T^{\;<v}(\overline{n})$, (a'')~$\forall y<\overline{n}\textsl{\textsf{Def}}_T^{\;<v}(y)$, (b')~$\neg\textsl{\textsf{Def}}_T^{\;<v}(u)$ and (b'')~$\forall y<u\textsl{\textsf{Def}}_T^{\;<v}(y)$ hold. Now, by $u\leqslant\overline{n}\vee\overline{n}\leqslant u$,  if $u\neq\overline{n}$ then either $u<\overline{n}$ or $\overline{n}<u$ holds. In the former case we have a contradiction between (a'') and (b'), and in the latter case we have a contradiction between (a') and (b''). Therefore, $\overline{n}=u$.
Now, assume (for the sake of contradiction) that $T\vdash
\textsl{\textsf{Boolos}}_T(\overline{\textsl{\textsf{b}}_T})$. Then   $T\vdash\forall u,v [\textsl{\textsf{Berry}}_T^{\;<v}(\overline{n})
\wedge\textsl{\textsf{Berry}}_T^{\;<v}(u)
\rightarrow \overline{n}=u]$, shown above,  implies the deduction    $T\vdash\forall x[\textsl{\textsf{Boolos}}_T(x)\leftrightarrow x=\overline{\textsl{\textsf{b}}_T}]$. Thus, $\textsl{\textsf{b}}_T$ is definable in $T$
by the formula $\textsl{\textsf{Boolos}}_T(x)$
whose length is less than
$\ell_T+\textsl{\textsf{len}}(\overline{5}\cdot\overline{\ell_T})
+9=4\ell_T+26<5\ell_T$ (since, for any $m$,
the term   $\overline{m}=\textsl{\textsf{s}}
(\cdots(\textsl{\textsf{s}}(0))\ldots)$
[$m$-times $\textsl{\textsf{s}}$] has length $3m+1$).
So, the $\Sigma_1$-sentence
$\textsl{\textsf{Def}}_T^{\;<\overline{5}
\cdot\overline{\ell_T}}(\overline{\textsl{\textsf{b}}_T})$ is true,
thus provable in $\textsl{\textsf{Q}}$;
whence $T\vdash
\textsl{\textsf{Def}}_T^{\;<\overline{5}
\cdot\overline{\ell_T}}(\overline{\textsl{\textsf{b}}_T})$.
On the other hand $T\vdash\textsl{\textsf{Boolos}}_T
(\overline{\textsl{\textsf{b}}_T})$
implies that $T\vdash\neg
\textsl{\textsf{Def}}_T^{\;<\overline{5}\cdot
\overline{\ell_T}}(\overline{\textsl{\textsf{b}}_T})$,
and this contradicts the consistency of $T$. 
\end{proof}

\noindent 
The formula $\textsl{\textsf{Boolos}}_T(\overline{\textsl{\textsf{b}}_T})$
is not $\Pi_1$; however, the following modification from \cite{kiku} proves a $\Pi_1$-incompleteness.

\begin{theorem}[Boolos' Theorem, modified]\label{boki}
For any consistent and {\sc re} extension $T$ of $\textsl{\textsf{Q}}$, the true $\Pi_1$-sentence $\neg
\textsl{\textsf{Def}}_T^{\;<\overline{5}
\cdot
\overline{\ell_T}}(\overline{\textsl{\textsf{b}}_T})$ is  unprovable in $T$.
\end{theorem}

\begin{proof}
Assume, to the contrary,  that $T\vdash\neg
\textsl{\textsf{Def}}_T^{\;<\overline{5}
\cdot
\overline{\ell_T}}(\overline{\textsl{\textsf{b}}_T})$. Since any $y<\textsl{\textsf{b}}_T$ is definable by a formula with length less than $5\ell_T$ then the $\Sigma_1$-sentence $\forall y<\textsl{\textsf{b}}_T\;
\textsl{\textsf{Def}}_T^{\;<\overline{5}
\cdot
\overline{\ell_T}}(y)$ is true and thus provable in $T$. Therefore, $\neg
\textsl{\textsf{Def}}_T^{\;<\overline{5}
\cdot
\overline{\ell_T}}(\overline{\textsl{\textsf{b}}_T})
\wedge \forall y<\textsl{\textsf{b}}_T\;
\textsl{\textsf{Def}}_T^{\;<\overline{5}
\cdot
\overline{\ell_T}}(y)$ is provable in $T$  and so $T\vdash\textsl{\textsf{Boolos}}_T
(\overline{\textsl{\textsf{b}}_T})$,  contradicting Theorem~\ref{both}.
\end{proof}

\noindent 
Even though $\ell_T$ is computable from $T$, below we show that one cannot calculate $\textsl{\textsf{b}}_T$.

\begin{theorem}[Non-Constructivity of Boolos'  Proof]\label{uncombo}
 There is no algorithm such
that for a given consistent and {\sc re} extension $T$ of $\textsl{\textsf{Q}}$ can compute $\textsl{\textsf{b}}_T$.
\end{theorem}

\begin{proof}
Assume that $\textsl{\textsf{b}}_T$ is computable from $T$, and let $T_0=\textsl{\textsf{Q}}$ and inductively $T_{j+1}=T_j+\neg
\textsl{\textsf{Def}}_{T_j}^{\;<\overline{5}\cdot
\overline{\ell_{T_j}}}
(\overline{\textsl{\textsf{b}}_{T_j}})$.
Define the function $\hbar(n)$, for any $n\in\mathbb{N}$, to be the maximum of $m$'s such that $\forall j<m: \textsl{\textsf{len}}(``\boldsymbol\varphi_j(0)\!\!\downarrow=x")<n$. This is a computable and non-decreasing function; also $\lim_n\hbar(n)=\infty$. So, from $\lim_j\ell_{T_j}=\infty$ we have $\lim_j\hbar({5}{\ell_{T_j}})=\infty$. Therefore, for any (given) $x$ one can compute some $\iota(x)$ such that $\hbar(5\ell_{T_{\iota(x)}})>x$. The proof will be complete when show that $\mathscr{K}(\textsl{\textsf{b}}_{T_{j}})
\geqslant\hbar(5\ell_{T_{j}})$ holds for any $j$: Because, by the computability of $\textsl{\textsf{b}}_{T_j}$ from $j$, we will have a computable function $x\mapsto\textsl{\textsf{b}}_{T_{\iota(x)}}$ which satisfies $\forall x: \mathscr{K}\big(\textsl{\textsf{b}}_{T_{\iota(x)}}\big)
\geqslant\hbar(5\ell_{T_{\iota(x)}})>x$ contradicting Lemma~\ref{mainlemma}.
For showing that $\mathscr{K}(\textsl{\textsf{b}}_{T_{j}})\geqslant\hbar(5\ell_{T_{j}})$ holds for any $j$, we show more generally that for any $u,v$ if $\neg
\textsl{\textsf{Def}}_T^{\;<v}(u)$ holds, for some consistent $T\supseteq\textsl{\textsf{Q}}$, then $\mathscr{K}(u)\geqslant\hbar(v)$: If, to the contrary, we have  $\mathscr{K}(u)<\hbar(v)$ then there exists some $j$ such that  (1)~$j<\hbar(v)$ and (2)~$\boldsymbol\varphi_j(0)\!\!\downarrow=u$. By (2) the number $u$ is definable by the formula $``\boldsymbol\varphi_j(0)\!\!\downarrow=x"$ in $\textsl{\textsf{Q}}$ (and so in $T$), and by (1) the length of the formula $``\boldsymbol\varphi_j(0)\!\!\downarrow=x"$ is less than $v$; so  $\textsl{\textsf{Def}}_T^{\;<v}(u)$ should hold, a contradiction.
\end{proof}

\noindent 
Of course, when $T$ is  $\Sigma_1$-sound then
$\neg
\textsl{\textsf{Def}}_T^{\;<\overline{5}
\cdot\overline{\ell_T}}(\overline{\textsl{\textsf{b}}_T})$ is   independent from $T$. Also  $\textsl{\textsf{Boolos}}_T(\overline{\textsl{\textsf{b}}_T})$ is independent from $T$: Because if $T\vdash\neg\textsl{\textsf{Boolos}}_T
(\overline{\textsl{\textsf{b}}_T})$ then $T\vdash\neg\textsl{\textsf{Berry}}_T^{\;<\overline{5}
\cdot
\overline{\ell_T}}
(\overline{\textsl{\textsf{b}}_T})$ and so  $T\vdash \forall y\!<\overline{\textsl{\textsf{b}}_T}\!
\;\textsl{\textsf{Def}}_T^{\;<\overline{5}\cdot
\overline{\ell_T}}(y)
\rightarrow
\textsl{\textsf{Def}}_T^{\;<\overline{5}\cdot
\overline{\ell_T}}(\overline{\textsl{\textsf{b}}_T})$. But $\forall y\!<\overline{\textsl{\textsf{b}}_T}\!\;
\textsl{\textsf{Def}}_T^{\;<\overline{5}\cdot
\overline{\ell_T}}(y)$, being a true $\Sigma_1$-sentence, is provable in $T$. Whence, we have  $T\vdash
\textsl{\textsf{Def}}_T^{\;<\overline{5}\cdot
\overline{\ell_T}}(\overline{\textsl{\textsf{b}}_T})$, a contradiction.
 However, we show in the following theorem that if $T$ is not $\Sigma_1$-sound then
$\textsl{\textsf{Def}}_T^{\;<\overline{5}\cdot
\overline{\ell_T}}(\overline{\textsl{\textsf{b}}_T})$, and so  $\neg\textsl{\textsf{Boolos}}_T(\overline{\textsl{\textsf{b}}_T})$, could be provable in $T$. For the following theorem to make sense we note that for any   theory $U$  satisfying the following conditions

\begin{itemize}
\item[(i)]  $U\nvdash\textsl{\textsf{Con}}(U)$,  i.e., G\"odel's Second Incompleteness Theorem holds for $U$;
\item[(ii)]  $U\vdash \textsl{\textsf{Con}}(U+\psi) \rightarrow \textsl{\textsf{Con}}(U)$, for any $\psi$;
\end{itemize}

\noindent
there exists a consistent theory $S\supseteq U$ such that $S+\textsl{\textsf{Con}}(S)$ is not consistent:
The theory $S=U+
\neg\textsl{\textsf{Con}}(U)$ is consistent
by  (i), and
$S\vdash\neg\textsl{\textsf{Con}}(S)$ because $S\vdash\neg\textsl{\textsf{Con}}(U)$ by the definition of $S$  and
$S\vdash \neg\textsl{\textsf{Con}}(U) \rightarrow \neg\textsl{\textsf{Con}}(S)$ by (ii).

\noindent 
One example for   a theory that satisfies the conditions (i) and (ii) above,  and also (iii) in Theorem~\ref{nrobo} and (iv)   in Theorem~\ref{boopt} below,    is Peano's Arithmetic. This arithmetic is indeed too strong and the finitely axiomatizable theory I$\Sigma_1$ (see~\cite{hp}) satisfies the conditions (i), (ii), (iii) and (iv). Even the weaker theories  I$\Delta_0+\Omega_1$ (see~\cite{wp}) and $\textsl{\textsf{S}}_2^1$ (see~\cite{buss}) are strong enough to satisfy them.

\begin{theorem}[Boolos' Proof is not Rosserian]\label{nrobo}
Suppose that a consistent and {\sc re} extension $T$ of $\textsl{\textsf{Q}}$   satisfies the following condition for any formula $\psi\!:$

\begin{itemize}
\item[\textup{(iii)}] $T\vdash \textsl{\textsf{Pr}}_T(\bot)\rightarrow \textsl{\textsf{Pr}}_T(\psi)$.
\end{itemize}

\noindent
If $T+\textsl{\textsf{Con}}(T)$ is \underline{in}consistent  then 
 for any $b\in\mathbb{N}$ we have
 $T\vdash\textsl{\textsf{Def}}_T^{\;<\overline{5}\cdot
\overline{\ell_T}}(\overline{b})$, and so  $T\vdash\neg\textsl{\textsf{Boolos}}_T(\overline{{b}})$.
\end{theorem}

\begin{proof}
If $T\vdash\neg\textsl{\textsf{Con}}(T)$ then $T\vdash\textsl{\textsf{Pr}}_T(\bot)$ and so $T\vdash\textsl{\textsf{Pr}}_T(\psi)$, for any  $\psi$, by the condition (iii). In particular, if $\psi$ is a formula with length less than $5\ell_T$ (for example $\textsl{\textsf{Berry}}^{\;<x'}(x)$) then $T\vdash\textsl{\textsf{Pr}}_T\big({\textsf{D}}
(\psi,b)\big)$ and so for any arbitrary number $b$ we have the following  deduction: $T\vdash\exists x
\big[\textsl{\textsf{len}}(x)<\overline{5}\cdot\overline{\ell_T}
\wedge \textsl{\textsf{Pr}}_T
\big({\textsf{D}}(x,b)\big)
\big]$. Thus $T\vdash
\textsl{\textsf{Def}}_T^{\;<\overline{5}
\cdot\overline{\ell_T}}(\overline{{b}})$, whence  $T\vdash\neg\textsl{\textsf{Berry}}_T^{\;<\overline{5}
\cdot\overline{\ell_T}}(\overline{{b}})$
and
$T\vdash\neg\textsl{\textsf{Boolos}}_T(\overline{{b}})$.
\end{proof}

\noindent 
So, if $T+\textsl{\textsf{Con}}(T)$ is not consistent, then $\neg
\textsl{\textsf{Def}}_T^{\;<\overline{5}
\cdot\overline{\ell_T}}(\overline{{b}})$  is not independent from the theory $T$ (neither is $\textsl{\textsf{Boolos}}_T(\overline{{b}})$) for any $b$. However, if $T+\textsl{\textsf{Con}}(T)$ is consistent, then a variant of Boolos' proof can go through (cf.~\cite[Theorem~7.2]{kita}) as is shown in the following  theorem.
\begin{theorem}[Boolos' Theorem, restated]\label{boopt}
If   an {\sc re} extension $T$ of $\textsl{\textsf{Q}}$ satisfies the following condition  for any $m,n,k\in\mathbb{N}$,

\begin{itemize}
\item[\textup{(iv)}]
$T\vdash \textsl{\textsf{Pr}}_T\big({\textsf{D}}
(k,m)\big) \wedge \textsl{\textsf{Pr}}_T\big({\textsf{D}}
(k,n)\big) \wedge \overline{m}\neq\overline{n} \rightarrow \neg\textsl{\textsf{Con}}(T)$,
\end{itemize} 

\noindent and the theory $T+\textsl{\textsf{Con}}(T)$ is consistent, then there exists some $b\in\mathbb{N}$ such that  $\neg\textsl{\textsf{Def}}_T^{\;<\overline{5}\cdot\overline{\ell_T}}
(\overline{{b}})$, and also $\textsl{\textsf{Boolos}}_T(\overline{{b}})$, is independent from $T$.
\end{theorem}

\begin{proof}
First we show that there exists some $a$ such that $T\nvdash\textsl{\textsf{Def}}_T^{\;<\overline{5}
\cdot\overline{\ell_T}}
(\overline{{a}})$. If not, then for any $i$  we have $T\vdash\textsl{\textsf{Def}}_T^{\;<\overline{5}
\cdot\overline{\ell_T}}
(\overline{{i}})$. Let $\Bbbk$ be a fixed number greater than the maximum G\"odel codes of formulas $\phi$ with  $\textsl{\textsf{len}}(\phi)<5\ell_T$. So, for any $i\leqslant\Bbbk$ we have the deduction  $T\vdash\exists z<\Bbbk\;\textsl{\textsf{Pr}}_T
\big({\textsf{D}}(z,{i})\big)$.
By Lemma~\ref{lemphp} there exists some $i<j\leqslant\Bbbk$ and some $\ell<\Bbbk$ such that $T\vdash\textsl{\textsf{Pr}}_T\big({\textsf{D}}
(\ell,{i})\big)\wedge\textsl{\textsf{Pr}}_T
\big({\textsf{D}}(\ell,{j})\big)$. Now (iv) implies that
 $T\vdash\neg\textsl{\textsf{Con}}(T)$,  a contradiction. Let $b$ be the minimum of those  $a$'s with $T\nvdash\textsl{\textsf{Def}}_T^{\;<\overline{5}
 \cdot\overline{\ell_T}}
(\overline{{a}})$. So, we have the deduction $T\vdash\forall z<\overline{b}\,\textsl{\textsf{Def}}_T^{\;<\overline{5}
 \cdot\overline{\ell_T}}
(z)$. Now we show that $T\nvdash\neg\textsl{\textsf{Def}}_T^{\;<\overline{5}
 \cdot\overline{\ell_T}}
(\overline{b})$: If not ($T\vdash\neg\textsl{\textsf{Def}}_T^{\;<\overline{5}
 \cdot\overline{\ell_T}}
(\overline{b})$) then $T\vdash
\textsl{\textsf{Berry}}_T^{\;<\overline{5}
 \cdot\overline{\ell_T}}
(\overline{b})$ or equivalently, $T\vdash\textsl{\textsf{Boolos}}_T(\overline{b})$. So, $b$ is definable in $T$ by a formula with length less than $5\ell_T$ (see the proof of Theorem~\ref{both}) whence $\textsl{\textsf{Def}}_T^{\;<\overline{5}
 \cdot\overline{\ell_T}}
(\overline{{b}})$ is true thus provable in $T$;  a contradiction. Therefore, we showed that $T\nvdash\textsl{\textsf{Def}}_T^{\;<\overline{5}
 \cdot\overline{\ell_T}}
(\overline{{b}})$ and $T\nvdash\neg\textsl{\textsf{Def}}_T^{\;<\overline{5}
 \cdot\overline{\ell_T}}
(\overline{{b}})$ (also $T\nvdash\textsl{\textsf{Boolos}}_T(\overline{{b}})$ and $T\nvdash\neg\textsl{\textsf{Boolos}}_T(\overline{{b}})$).
\end{proof}

\noindent 
Thus, the consistency of $T+\textsl{\textsf{Con}}(T)$ is an optimal (indeed, necessary and sufficient) condition for the independence of a Boolos sentence from $T$.

\section{Concluding Remarks}
The following table summarizes some of the new and old results in this paper:
\begin{center}
\begin{tabular}{|c||c|c|}
\hline
Proof & ~ \quad ~   Constructive ~ \quad ~ & ~ \;  ~ Rosser Property ~ \;  ~ \\
\hline
\hline
{\sc G\"odel}~(1931) \hfill \cite{godel}  &  $\checkmark$  \hfill  ~ & {\scriptsize  \textbf{\textsf{X}}} \hfill \cite{isaac}  \\
\hline
{\sc Rosser}~(1936) \hfill \cite{rosser} & $\checkmark$ \hfill  ~ & $\checkmark$ \hfill  ~ \\
\hline
{\sc Kleene}$_1$~(1936) \hfill \cite{kleene1}  & $\checkmark$ \hfill  ~ & {\scriptsize  \textbf{\textsf{X}}} \hfill  Theorem~\ref{nrk} \\
\hline
{\sc Kleene}$_2$~(1950) \hfill \cite{kleene2}  & $\checkmark$ \hfill  ~  &  $\checkmark$ \hfill  ~  \\
\hline
{\sc Chaitin}~(1971) \hfill \cite{chai} &  {\scriptsize  \textbf{\textsf{X}}} \hfill \cite{lambal,stillwell},\! Theorem~\ref{unconch} & $\checkmark$ \hfill Theorem~\ref{rosch} \\
\hline
{\sc Boolos}~(1989) \hfill \cite{boolos} &  {\scriptsize  \textbf{\textsf{X}}} \hfill Theorem~\ref{uncombo}   & {\scriptsize  \textbf{\textsf{X}}} \hfill Theorem~\ref{nrobo}  \\
\hline
\end{tabular}
\end{center}
Let us note that for the constructivity of a proof, usually, no new argument is needed as a computational procedure could often be seen from the proof. But the non-cosntructivity of a proof (as in the case of Chaitin's and Boolos' proofs) should be proved; proving the non-constructivity (the non-existence of any algorithm) is usually harder than showing the constructivity (the existence of an algorithm). So is having the Rosser property of a proof. Other than Rosser's proof and Kleene's symmetric theorem (1950) Chaitin's proof has the Rosser property. The non-Rosserian  proofs of G\"odel and Boolos need the consistency of $T+\textsl{\textsf{Con}}(T)$ for the independence of their true but unprovable sentences, and this condition, $\textsl{\textsf{Con}}\big(T+\textsl{\textsf{Con}}(T)\big)$, is optimal (for the independence of that sentences).

\paragraph{\textsl{\textbf{Acknowledgements}.}}
The authors warmly thank Professor Albert Visser for carefully reading the  paper and for his comments   which greatly improved the results of the paper and its presentation and notation.
This paper (except the second section on  Kleene's proof) is a part of the Ph.D. thesis of the second author written at  the University of Tabriz, {\sc Iran},   under the supervision of the first author who is partially supported by grant
$\textsf{N}^{\underline{\tt o}}$~94030033 from  the Institute for Research in Fundamental Sciences
  $\bigcirc\hspace{-2ex}{\not}\hspace{0.325ex}\bullet
  \hspace{-1.4ex}{\not}\hspace{-0.25ex}\bigcirc$
  $\mathbb{I}\mathbb{P}\mathbb{M}$.



\begin{thebibliography}{999}

\bibitem[Beklemishev (2010)]{bek}
{\sc Beklemishev, L.~D.},
\newblock{``G\"odel Incompleteness Theorems and the Limits of their Applicability,~I''}, \newblock{\em Russian Mathematical Surveys}, vol.~65 (2010), pp.~857--899.  \textsc{doi}:~
\href{http://dx.doi.org/10.1070/RM2010v065n05ABEH004703}
{10.1070/RM2010v065n05ABEH004703}.

\bibitem[Boolos (1989)]{boolos}
{\sc Boolos, G.},
\newblock{``A New Proof of the G\"odel Incompleteness Theorem''}, \newblock{\em Notices of the American Mathematical Society}, vol.~36 (1989), pp.~388--390.  Reprinted in  {\sc Boolos, G.}, {\em Logic, Logic, and Logic} ({\sc isbn}:~\href{http://www.isbnsearch.org/isbn/9780674537675}
{9780674537675}), Harvard Unviesity Press (1998), pp.~383--388.


\bibitem[Buss (1986)]{buss} {\sc Buss, S.},
\newblock{\em Bounded Arithmetic},
\newblock{Bibliopolis, Naples, Italy  (1986)}. \\
{\sc url}:~\href{https://www.math.ucsd.edu/~sbuss/ResearchWeb/BAthesis/}
{https://www.math.ucsd.edu/$\sim$sbuss/ResearchWeb/BAthesis/}.



\bibitem[Chaitin (1971)]{chai}
{\sc 	 Chaitin, G.~J.},
\newblock{``Computational Complexity and G\"odel's Incompleteness Theorem''}, \newblock{\em SIGACT News}, vol.~9 (1971), pp.~11--12.
\textsc{doi}:~\href{http://dx.doi.org/10.1145/1247066.1247068}
{10.1145/1247066.1247068}.



\bibitem[Davis (1978)]{davis}
{\sc Davis, M.},
\newblock{``What is a Computation?''},   \newblock{\em Mathematics Today, Twelve Informal Essays}, L.~A.~Steen (ed.),  Springer (1978), pp.~241--267.  \textsc{doi}:~\href{http://dx.doi.org/10.1007/978-1-4613-9435-8_10}
{10.1007/978-1-4613-9435-8\_10}.



\bibitem[Godel (1931)]{godel}
{\sc G\"odel, K.},
\newblock{``\"{U}ber formal unentscheidbare S\"{a}tze der Principia Mathematica und verwandter Systeme,~I.''},
\newblock{\em Monatshefte f\"{u}r Mathematik und Physik}. vol.~38 (1931), pp.~173--198.
\textsc{doi}:~\href{http://dx.doi.org/10.1007/BF01700692}
{10.1007/BF01700692}. \\
Translated as
\newblock{``On Formally Undecidable Propositions of {\em   Principia Mathematica} and Related Systems,~I.''},   \newblock{\em Kurt G\"odel Collected Works, Volume~I: Publications~ 1929--1936} ({\sc isbn}:~\href{http://www.isbnsearch.org/isbn/9780195039641}
{9780195039641}), S.~Feferman et al. (eds.),  Oxford University Press (1986), pp.~135--152.

\bibitem{hp}
{\sc  {H\'{a}jek}, P.}  and  {\sc {Pudl\'{a}k}, P.},
\newblock {\em Metamathematics
of First-Order Arithmetic},
\newblock Springer-Verlag, 2nd. print, 1998.
{\sc isbn}:~\href{http://www.isbnsearch.org/isbn/9783540636489}
{9783540636489}.  \\
\url{http://projecteuclid.org/euclid.pl/1235421926}.


\bibitem[Isaacson (2011)]{isaac}
{\sc Isaacson, D.},
\newblock{``Necessary and Sufficient Conditions for Undecidability of the G\"odel Sentence and its Truth''},   \newblock{\em Logic, Mathematics, Philosophy: Vintage Enthusiasms}, D.~DeVidi et al. (eds.),  Springer (2011), pp.~135--152.  \textsc{doi}:~\href{http://dx.doi.org/10.1007/978-94-007-0214-1_7}
{10.1007/978-94-007-0214-1\_7}.


\bibitem[Kikuchi(1994)]{kiku}
{\sc Kikuchi, M.},
\newblock{``A Note on Boolos' Proof of the Incompleteness Theorem''},
\newblock {\em Mathematical Logic Quarterly}, vol.~40 (1994),   pp.~528--532. \textsc{doi}:~\href{http://dx.doi.org/10.1002/malq.19940400409}
{10.1002/malq.19940400409}.



\bibitem[Kikuchi \& Tanaka (1994)]{kita}
{\sc Kikuchi, M.} and {\sc Tanaka, K.},
\newblock{``On Formalization of Model-Theoretic Proofs of G\"odel's Theorems''},
\newblock {\em Notre Dame Journal of Formal Logic}, vol.~35 (1994),   pp.~403--412. \textsc{doi}:~\href{http://dx.doi.org/10.1305/ndjfl/1040511346}
{10.1305/ndjfl/1040511346}.


\bibitem[Kikuchi \& Kurahashi \& Sakai (1994)]{kikusa}
{\sc Kikuchi, M.} and {\sc Kurahashi, T.} and {\sc Sakai, H.},
\newblock{``On Proofs of the Incompleteness
Theorems Based on Berry's Paradox by Vop\v{e}nka, Chaitin, and Boolos''},
\newblock {\em Mathematical Logic Quarterly}, vol.~58 (2012),   pp.~307--316. \textsc{doi}:~\href{http://dx.doi.org/10.1002/malq.201110067}
{10.1002/malq.201110067}.


\bibitem[Kleene (1936)]{kleene1}
{\sc Kleene, S.~C.},
\newblock{``General Recursive Functions of Natural Numbers''},
\newblock {\em Mathematische Annalen}, vol.~112 (1936),   pp.~727--742.
{\sc doi}:~\href{http://dx.doi.org/10.1007/BF01565439}
{10.1007/BF01565439}.



\bibitem[Kleene (1950)]{kleene2}
{\sc Kleene, S.~C.},
\newblock{``A Symmetric Form of G\"odel's Theorem''},
\newblock {\em Indagationes Mathem  aticae} ({\sc issn}:~\href{https://www.journals.elsevier.com/indagationes-mathematicae}
{0019-3577}), vol.~12 (1950),   pp.~244--246.


\bibitem[Kleene (1952)]{kleene3}
{\sc Kleene, S.~C.},
\newblock {\em Introduction to Metamathematics}, 
\newblock  {North-Holland (1952)}.
{\sc isbn}:~\href{http://www.isbnsearch.org/isbn/9780720421033}
{9780720421033}.




\bibitem[Kotlarski (2004)]{kotlar}
{\sc Kotlarski, H.},
\newblock{``The Incompleteness Theorems After 70 Years''},
\newblock {\em Annals of Pure and Applied Logic}, vol.~126 (2004),   pp.~125--138. \textsc{doi}:~\href{http://dx.doi.org/10.1016/j.apal.2003.10.012}
{10.1016/j.apal.2003.10.012}.



\bibitem[Lambalgen (1989)]{lambal}
{\sc van Lambalgen, M.},
\newblock{``Algorithmic Information Theory''},
\newblock {\em Journal of Symbolic Logic}, vol.~54 (1989),   pp.~1389--1400. \textsc{doi}:~\href{http://dx.doi.org/10.1017/S0022481200041153}
{10.1017/S0022481200041153}.



\bibitem[Rosser (1936)]{rosser}
{\sc Rosser, B.},
\newblock{``Extensions of Some Theorems of G\"odel and Church''},
\newblock {\em Journal of Symbolic Logic}, vol.~1 (1936),   pp.~87--91. \textsc{doi}:~\href{http://dx.doi.org/10.2307/2269028}
{10.2307/2269028}.



\bibitem[Roy (2003)]{roy}
{\sc Roy, D.~K.},
\newblock{``The Shortest Definition of a Number in Peano Arithmetic''},
\newblock {\em Mathematical Logic Quarterly}, vol.~49 (2003),   pp.~83--86. \textsc{doi}:~\href{http://dx.doi.org/10.1002/malq.200310006}
{10.1002/malq.200310006}.



\bibitem[Salehi (2014)]{salehi}
{\sc Salehi, S.},
\newblock{``G\"odel's Incompleteness Phenomenon---Computationally''},
\newblock {\em Philosophia Scienti\ae}, vol.~18 (2014),   pp.~23--37. \\ {\sc doi}:~\href{http://dx.doi.org/10.4000/philosophiascientiae.968}
{10.4000/philosophiascientiae.968}.




\bibitem[Ser\'{e}ny (2004)]{sere}
{\sc Ser\'{e}ny, G.},
\newblock{``Boolos-Style Proofs of Limitative Theorems''},
\newblock {\em Mathematical Logic Quarterly}, vol.~50 (2004),   pp.~211--216. \textsc{doi}:~\href{http://dx.doi.org/10.1002/malq.200310091}
{10.1002/malq.200310091}.






\bibitem[Smith (2013)]{smith} {\sc Smith, P.},
\newblock{\em An Introduction to G\"odel's Theorems}, 2nd edition,
\newblock{Cambridge University Press (2013)}.
{\sc isbn}:~\href{http://www.isbnsearch.org/isbn/9781107606753}
{9781107606753}.



%
%

\bibitem[Smorynski (1977)]{smor}
{\sc Smory\'{n}ski, C.},
\newblock{``The Incompleteness Theorems''},
\newblock{\em Handbook of Mathematical Logic}, J.~Barwise (ed.),  North-Holland (1977), {\sc isbn}:~\href{http://www.isbnsearch.org/isbn/9780444863881}
{9780444863881}, pp.~821--865.  


\bibitem[Stillwell (2010)]{stillwell}
{\sc Stillwell, J.~C.},
\newblock {\em Roads to Infinity: the mathematics of truth and proof}, 
\newblock  {A K Peters / CRC Press (2010)}.
{\sc isbn}:~\href{http://www.isbnsearch.org/isbn/9781568814667}
{9781568814667}.


\bibitem[Tarski \& Mostowski \& Robinson (1953)]{tmr} {\sc  Tarski, A.},  and   {\sc Mostowski, A.} and   {\sc Robinson, R.M.},
 {\em Undecidable Theories}, North--Holland (1953), reprinted by Dover Publications (2010).
 {\sc isbn}:~\href{http://www.isbnsearch.org/isbn/9780486477039}
{9780486477039}.


\bibitem[Wilkie \& Paris (1987)]{wp}
{\sc Wilkie, A.J.} and {\sc Paris, J.B.},
\newblock{``On the Scheme of Induction for Bounded Arithmetic Formulas''},
\newblock {\em Annals of Pure and Applied Logic}, vol.~35 (1987),   pp.~261--302. \textsc{doi}:~\href{http://dx.doi.org/10.1016/0168-0072(87)90066-2}
{10.1016/0168-0072(87)90066-2}.





\end{thebibliography}
\end{document}